\documentclass[11pt,a4paper]{article}
\usepackage[utf8]{inputenc}

\usepackage{amsmath,amsthm,amsfonts,amssymb}
\usepackage{graphicx,url}

\usepackage[left=2.5cm,right=2.5cm,top=3cm,bottom=3cm]{geometry}
\usepackage{setspace}
\renewcommand{\geq}{\geqslant}
\renewcommand{\leq}{\leqslant}
\renewcommand{\ge}{\geqslant}
\renewcommand{\le}{\leqslant}

\def\eref#1{$(\ref{#1})$}

\def\lref#1{Lemma~$\ref{#1}$}
\def\tref#1{Theorem~$\ref{#1}$}

\def\cref#1{Corollary~$\ref{#1}$}

\author{Aidan R. Gentle\\
\small School of Mathematics\\[-0.5ex]
\small Monash University\\[-0.5ex]
\small Vic 3800, Australia\\
\small\tt aidan.gentle@monash.edu}
\title{A polynomial construction of Perfect Sequence Covering Arrays}

\date{}

\newtheorem{theorem}{Theorem}[section]
\newtheorem{lemma}[theorem]{Lemma}

\theoremstyle{definition}

\theoremstyle{definition}

\newcommand{\sym}{\mathcal{S}}
\newcommand{\stab}{\textup{Stab}}
\newcommand{\orb}{\textup{Orb}}
\newcommand{\asc}{\textup{Asc}}
\newcommand{\pgaml}{\textup{P}\Gamma\textup{L}}

\begin{document}
\maketitle

\begin{abstract}
  A PSCA$(v, t, \lambda)$ is a multiset of permutations of the
  $v$-element alphabet $\{0, \dots, v-1\}$ such that every sequence of
  $t$ distinct elements of the alphabet appears in the specified order
  in exactly $\lambda$ permutations. For $v \geq t$,
  let $g(v, t)$ be the smallest positive integer $\lambda$ such
  that a PSCA$(v, t, \lambda)$ exists. We present an explicit construction
  that proves $g(v,t) = O(v^{t(t-2)})$ for fixed $t \ge 4$. The method 
  of construction involves taking a permutation representation of the group 
  of projectivities of a suitable projective space of dimension $t - 2$ and 
  deleting all but a certain number of symbols from each permutation. In the 
  case that this space is a Desarguesian projective plane, we also show that 
  there exists a permutation representation of the group of projectivities of 
  the plane that covers the vast majority of 4-sequences of its points a 
  fixed number of times.
\end{abstract}



\section{Introduction}\label{s:intro}

For positive integers $v$ and $t$ with $v \geq t$, let $[v] = \{ 0, \dots, v-1 \}$, $\sym_{v}$ be the group of permutations of $[v]$, and $\sym_{v, t}$ be the set of ordered sequences of $t$ distinct elements of $[v]$. Unless stated otherwise, the elements of a sequence $s \in \sym_{v,t}$ are denoted by $(s_{1}, \dots, s_{t})$. For $\pi \in \sym_{v}$ and $s \in \sym_{v, t}$ we say that $\pi$ \emph{covers} $s$ if $\pi^{-1}(s_{i}) < \pi^{-1}(s_{i + 1})$ for $i \in \{1, \dots, t-1\}$. Several aspects of sequence covering have been studied including the problems of finding packings, coverings and perfect coverings of sequences. In this context, a packing is a set of permutations in $\sym_{v}$ such that every sequence in $\sym_{v,t}$ is covered by at most one permutation. In coding theory, these sets are referred to as \textit{ $(v-t)$-deletion correcting codes} \cite{Kle04, Lev91}. A covering of sequences is a set of permutations in $\sym_{v}$ such that every sequence in $\sym_{v,t}$ is covered by at least one permutation. These sets are referred to as \textit{sequence covering arrays} and were first studied by Spencer \cite{Spen71} as an extension of a problem studied by Dushnik \cite{Dush50} relating to the dimension of certain partial orders. More recently, sequence covering arrays have been investigated for their applications in event sequence testing \cite{Kuhn12}. 

In this paper, we are concerned with the problem of perfect coverings of sequences. A \emph{perfect sequence covering array} with order $v$, strength $t$ and multiplicity $\lambda$, denoted by PSCA$(v, t, \lambda)$, is a multiset $X$ of permutations in $\sym_{v}$ such that every sequence in $\sym_{v, t}$ is covered by exactly $\lambda$ permutations in $X$. If $T$ is a $t$-subset of $[v]$, then there are $t!$ ways of arranging the elements of $T$, each of which forms a sequence in $\sym_{v,t}$ that must be covered by $\lambda$ permutations in a PSCA$(v, t, \lambda)$. Furthermore, every permutation in a PSCA$(v, t, \lambda)$ covers exactly one of these sequences, so a PSCA$(v, t, \lambda)$ must contain $t!\lambda$ permutations.

For $v \geq t$, let $g(v, t)$ be the smallest positive integer $\lambda$ such that a PSCA$(v, t, \lambda)$ exists. Observe that $\sym_{v}$ is a PSCA$(v, t, v!/t!)$, so $g(v, t)$ is well defined and $g(v, t) \leq v!/t!$. Much of the research into perfect sequence covering arrays has focussed on determining or bounding $g(v,t)$. If $v > t$ and we remove the symbol $v-1$ from every permutation of a PSCA$(v, t, \lambda)$, then we obtain a PSCA$(v-1, t, \lambda)$. Hence, $g(v, t) \geq g(v-1, t)$. For $2 \leq t' \leq t$, a PSCA$(v, t, \lambda)$ is also a PSCA$(v, t', \lambda \binom{t}{t'})$, so $g(v, t') \leq \binom{t}{t'}g(v, t)$.

In this paper, we present an explicit construction of a PSCA$(v,t,\lambda)$ for all $v \geqslant t \geqslant 4$. The method of this construction involves taking a suitable permutation representation of the group $\textup{PGL}(t-1,q)$ of projectivities of the projective space $\textup{PG}(t-2,q)$. We show that in such a permutation representation, there is a subset of $q+1$ symbols such that any $t$-sequence of symbols from this subset is covered by $\lambda$ permutations for a given constant $\lambda$. Hence, deleting all but the symbols in this subset forms a PSCA$(q+1,t,\lambda)$. This construction yields the following upper bound on $g(v,t)$.
\begin{theorem}\label{t:upperbound}
For $v \ge t \ge 4$,
\begin{equation*}
g(v, t) <  \frac{(2v)^{(t-1)^{2}}}{t!(2v - 1)}.
\end{equation*}
\end{theorem}

This is the first upper bound on $g(v,t)$ for $t \geqslant 4$ that is polynomial in $v$ and does not use probabilistic arguments. However, a probabilistic upper bound does exist. A \textit{$t$-wise uniform set of permutations} is a set $T \subseteq \sym_{v}$ such that for any $u,v \in \sym_{v,t}$,
\begin{equation*}
\frac{1}{\vert T \vert} \vert \{ \pi \in T : \pi(u_{i}) = v_{i}, 1\le i \le t\} \vert = \frac{t!}{n!}.
\end{equation*}
A $t$-wise uniform set of permutations $T \subseteq \sym_{v}$ is also a PSCA$(v,t,\vert T \vert/t!)$. Kuperberg, Lovett and Peled \cite{Kup17} proved that for any $t \le v$, there is a $t$-wise uniform set of permutations $T \subseteq \sym_{v}$ with $\vert T \vert \le (cv)^{ct}$ for some universal constant $c > 0$. Although this result gives a tighter bound on $g(v,t)$ than \tref{t:upperbound}, it is not yet known how to efficiently construct either a PSCA or a $t$-wise uniform set of permutations with this size. The best known construction of a $t$-wise uniform set of permutations is due to Finucane, Peled and Yaari \cite{Fin13} who build such a set $T \subseteq \sym_{v}$ with $\vert T \vert = t^{2v}$ when $v = 2^{m}$ and $t = 2^{\ell} - 1$ for some $m \geq 1$ and $\ell \geq 2$.

Our construction can also be applied when $t = 3$. However, an infinite family of PSCA$(v,3,\lambda)$ built by Yuster \cite{Yus19} established that $g(v,3) \leq cv(\log v)^{\log 7}$ for an absolute constant $c$. This result provides a tighter bound on $g(v,3)$ than \tref{t:upperbound} would, were it to be extended to the $t=3$ case.

In proving \tref{t:upperbound}, we show that sequences of $t$ points of $\textup{PG}(t-2,q)$ belonging to a particular family are covered by a constant number of permutations in a representation of $\textup{PGL}(t-1,q)$. In the $t=4$ case, we can choose a particular representation of $\textup{PGL}(3,q)$ to ensure that sequences of four points of $\textup{PG}(2,q)$ belonging to a separate family are also covered by the same constant number of permutations. Although accounting for this new family of sequences does not provide a substantial improvement to the bound $g(v,4) = O(v^{8})$ implied by \tref{t:upperbound}, it does prove the following theorem.

\begin{theorem}\label{t:collineationgroup}
Let $q$ be a prime power and let $r = q^{2} + q + 1$. Then there is a permutation representation $\Psi \leq \sym_{r}$ of $\textup{PGL}(3,q)$ such that the number of sequences in $\sym_{r,4}$ that are covered by exactly $\vert \Psi \vert/4!$ permutations in $G$ is greater than
\begin{equation*}
\left( 1 - \frac{1}{q} \right) \vert \sym_{r,4} \vert.
\end{equation*}
\end{theorem}

In building a PSCA$(v,4,\lambda)$, we take a prime power $q$ such that $q \geqslant v$, find a suitable permutation representation of $\textup{PGL}(3,q)$ in $\sym_{q^{2} + q + 1}$ and then delete all but $v$ symbols from each permutation in this group. As a consequence of this symbol deletion, the number of permutations in the resulting PSCA on $v$ symbols is approximately $v^{8}$. \tref{t:collineationgroup} implies that it is possible to find a set of permutations in $\sym_{v}$ with size approximately $v^{4}$ such that the vast majority of 4-sequences are covered by a constant number of permutations. This reduced size is much closer to the lower bound proved by Yuster \cite{Yus19}, which says that $g(v,4) \geq v(v-3)/48$. However, as the permutation representation presented in \tref{t:collineationgroup} may or may not be a PSCA, it is still unclear what the asymptotic behaviour of $g(v,4)$ should be.

In addition to the asymptotic results cited above, research in this area has uncovered exact values of $g(v,t)$ \cite{GeWa22, Lev91, Math99, Njl22, Yus19} as well as the non-existence of a PSCA$(v,t,\lambda)$ \cite{Chee13, GeWa22, Kle04, Math99} for certain choices of $v$, $t$ and $\lambda$.

The paper is organised as follows. In Section 2, we introduce notation and some basic ideas that lay the groundwork for the constructions in the subsequent sections. In Section 3, we prove \tref{t:upperbound}. In Section 4, we prove \tref{t:collineationgroup}.

\section{Preliminaries}\label{s:generous}

We begin by recalling some definitions regarding group actions. For a set $X$, let $\text{Sym}(X)$ denote the group of permutations of $X$. Note that when $X = [v]$, $\text{Sym}(X) = \sym_{v}$. An \emph{action} of a group $G$ on $X$ is a homomorphism $\phi: G \rightarrow \text{Sym}(X)$. For $g \in G$ and $x \in X$, we use $gx$ to refer to the image of $x$ under the permutation $\phi(g)$. The \textit{orbit} of $x$ is the set $\orb(x) := \{ gx : g \in G \}$. The \textit{stabiliser} of $x$ is the set $\stab(x) := \{ g \in G : gx = x \}$. The stabiliser of $x$ forms a subgroup of $G$. In what follows, we make use of the Orbit-Stabiliser Theorem.

\begin{theorem}
If $G$ is a group acting on $X$, then for any $x \in X$, 
\begin{equation*}
\vert \orb(x) \vert \vert \stab(x) \vert = \vert G \vert.
\end{equation*}
\end{theorem}

A permutation group $G \le \sym_{v}$ has the following natural action on $\sym_{v,t}$. If $g \in G$ and $s \in \sym_{v,t}$, then $gs = (g(s_{1}), \dots, g(s_{t}))$. Consider an array $\textsf{A}$ with columns indexed by $[v]$ and rows indexed by the elements of $G$ where $\textsf{A}[g,i] = g(i)$. Let $s \in \sym_{v,t}$ and consider the corresponding sequence of columns of \textsf{A}. In row $g$ and in columns $(s_{1}, \dots, s_{t})$ of \textsf{A}, we find the sequence $(g(s_{1}), \dots, g(s_{t})) = gs$. So the sequences that appear in the columns $(s_{1}, \dots, s_{t})$ of \textsf{A} are exactly those in $\orb(s)$. For $x \in \orb(s)$, the set of rows of \textsf{A} in which the sequence $x$ appears in the columns $(s_{1}, \dots, s_{t})$ is $\{ g : gs = x \}$. This set is a coset of $\stab(s)$ so it must have the same size as $\stab(s)$. The permutation $g$ covers $gs$ if and only if $s_{1} < \dots < s_{t}$. Define $\asc(s) := \{ x \in \orb(s) : x_{1} < \dots < x_{t} \}$. We now have the following lemma.
\begin{lemma}\label{l:ascstab}
If $G \le \sym_{v}$ is a permutation group and $s \in \sym_{v, t}$, then the number of permutations in $G$ that cover $s$ is $\vert \asc(s) \vert \vert \stab(s) \vert$.

\end{lemma} 

To conclude this section, we consider deleting symbols from permutations. For $\pi \in \sym_{v}$ and $j \in [v]$, we define $\pi_{[j]}$ to be the permutation in $\sym_{j}$ obtained by deleting the symbols $\{j, j+1, \dots, v-1\}$ from $\pi$. The permutation $\pi_{[j]}$ covers a sequence $s \in \sym_{j, t}$ if and only if $\pi$ also covers $s$. If $X$ is a multiset of permutations in $\sym_{v}$, then for $j \in [v]$, we define $X_{[j]}$ to be the multiset $\{ \pi_{[j]} : \pi \in X\}$. Then, for any sequence $s \in \sym_{j,t}$, the number of permutations in $X_{[j]}$ that cover $s$ is equal to the number of permutations in $X$ that cover $s$. In the next section we construct a PSCA of strength $t$ by deleting symbols from a suitable multiset of permutations.

\section{Collineations of projective spaces} \label{s:collineation}

In this section we prove \tref{t:upperbound}. We begin by introducing some definitions regarding projective spaces. Let $q = p^{m}$ for some prime $p$ and for some integer $m \geq 1$ and let $\textrm{GF}(q)$ be the field with $q$ elements. Now let $n \geq 2$ and let $V$ be an $(n+1)$-dimensional vector space over $\textrm{GF}(q)$. Then the \textit{$n$-dimensional projective space} over $\textrm{GF}(q)$, denoted by $\textrm{PG}(n,q)$ is the set of all 1-dimensional subspaces of $V$. The elements of $\textrm{PG}(n,q)$ are called \textit{points}. If $W$ is a subspace of $V$, then $W$ forms a set of points in $\textrm{PG}(n,q)$ with $W$ containing the point $X$ if $X$ is a subspace of $W$. A 2-dimensional subspace of $V$ forms a \textit{line} in $\textrm{PG}(n,q)$ and an $n$-dimensional subspace of $V$ forms a \textit{hyperplane} in $\textrm{PG}(n,q)$. A \textit{collineation} of $\textrm{PG}(n,q)$ is a permutation of the points of $\textrm{PG}(n,q)$ that maps lines to lines. Let $A \in \textrm{GL}(n+1,q)$ be a non-singular matrix and suppose that $Au = v$ for vectors $u \in X$ and $v \in Y$ where $X$ and $Y$ are points in $\textrm{PG}(n,q)$. Then $A(cu) = cv$ for $c \in \textrm{GF}(q)$. Thus, every vector in $X$ is mapped by $A$ to a vector in $Y$. Hence, $A$ induces a permutation of the points of $\textrm{PG}(n,q)$. Permutations formed in this way are called \textit{projectivities}. The set of all projectivities of $\textrm{PG}(n,q)$ forms the group $\textrm{PGL}(n+1,q)$. A \textit{frame} of $\textrm{PG}(n,q)$ is an ordered sequence of $n + 2$ points in $\textrm{PG}(n,q)$ such that no $n+1$ of these points lie in the same hyperplane of $\textrm{PG}(n,q)$. The following theorem is a statement of the Fundamental Theorem of Projective Geometry (see e.g. \cite{Hir79}).

\begin{theorem}\label{t:fundamental}
For any two frames in $\textup{PG}(n,q)$, there is a unique projectivity of $\textup{PG}(n,q)$ mapping one frame to the other.
\end{theorem}

Let $r$ be the number of points in $\textrm{PG}(n,q)$. Projectivities are defined as permutations of the points of $\textrm{PG}(n,q)$ but we can view projectivities as permutations of $[r]$ by labelling the points of $\textrm{PG}(n,q)$. For a bijection $\psi: \textrm{PG}(n,q) \rightarrow [r]$ and a projectivity $f$, define $f_{\psi} \in \sym_{r}$ by $f_{\psi}(i) = \psi(f(\psi^{-1}(i)))$. Then let $\Psi := \{ f_{\psi} : f \in \textrm{PGL}(n+1,q) \}$. Note $\Psi$ is a permutation subgroup of $\sym_{r}$. The order of $\Psi$ is given by
\begin{equation*}
\vert \Psi \vert = \vert \textrm{PGL}(n+1,q)\vert = \frac{\prod_{i=0}^{n} (q^{n+1} - q^{i})}{q-1}.
\end{equation*}
By establishing the bijection $\psi$, points of $\textrm{PG}(n,q)$ are associated with elements of $[r]$ and so we can treat lines and hyperplanes as subsets of $[r]$ and frames as sequences in $\sym_{r,n+2}$.

\begin{lemma}\label{l:framecover}
For a bijection $\psi: \textup{PG}(n,q) \rightarrow [r]$ and a frame $s \in \sym_{r,n+2}$, the number of permutations in $\Psi$ that cover $s$ is $\vert \Psi \vert / (n+2)!$. 
\end{lemma}

\begin{proof}
By \tref{t:fundamental}, every frame in $\sym_{r,n+2}$ is part of the same orbit under the action of $\Psi$. For any frame $s \in \sym_{r,n+2}$, any reordering of the points of $s$ will form another frame. Of all the $(n+2)!$ ways of ordering the points of $s$, only one of these sequences is in ascending order. Thus, $\vert \asc(s) \vert = \vert \orb(s) \vert/(n+2)!$. Therefore, by \lref{l:ascstab}, the number of permutations in $\Psi$ that cover $s$ is
\begin{equation*}
\vert \stab(s) \vert \vert \asc(s) \vert = \frac{\vert \stab(s) \vert \vert \orb(s) \vert}{(n+2)!} = \frac{\vert \Psi \vert}{(n+2)!}. \qedhere
\end{equation*}
\end{proof}

A \textit{$k$-arc} in $\textrm{PG}(n,q)$ is a set of $k$ points in $\textrm{PG}(n,q)$, no $n+1$ of which lie in a hyperplane of $\textrm{PG}(n,q)$. 

\begin{theorem}[e.g.~\cite{Ball19}]\label{t:arcsize}
For $q \geq n$, there exists a $(q+1)$-arc in $\textup{PG}(n,q)$.
\end{theorem}

\begin{lemma}\label{l:construction}
For $q \geq n+1$, if $\psi: \textup{PG}(n,q) \rightarrow [r]$ is a bijection such that $\{ \psi^{-1}(i) : i \in [q+1] \}$ is a $(q+1)$-arc, then $\Psi_{[q+1]}$ is a $\textup{PSCA}(q+1, n+2, \vert \Psi \vert/(n+2)!)$.
\end{lemma}

\begin{proof}
Let $s \in \sym_{q+1, n+2}$ and let $\psi$ be as defined in the lemma statement. Then $s$ is a frame. Thus, by \lref{l:framecover}, $s$ is covered by $\vert \Psi \vert/(n+2)!$ permutations in $\Psi$. Hence, $s$ is covered by $\vert \Psi \vert/(n+2)!$ permutations in $\Psi_{[q+1]}$. Therefore, $\Psi_{[q+1]}$ is a PSCA$(q+1, n+2, \vert \Psi \vert/(n+2)!)$.
\end{proof}

Note that for $q \geq n+1$, \tref{t:arcsize} guarantees the existence of a bijection $\psi$ satisfying the condition of \lref{l:construction}. We are now ready to prove \tref{t:upperbound}

\begin{proof}[Proof of Theorem \ref{t:upperbound}]
Let $q$ be the smallest power of 2 such that $q \geq v$. Then $q < 2v$. Let $n = t - 2$. Then, by \lref{l:construction},
\begin{equation*}
g(v,t) \leq \frac{\vert \textrm{PGL}(n+1,q) \vert}{(n+2)!} < \frac{\vert \textrm{PGL}(t-1,2v) \vert}{t!} = \frac{\prod_{i=0}^{t-2}((2v)^{t-1} - (2v)^{i})}{t!(2v - 1)} < \frac{(2v)^{(t-1)^{2}}}{t!(2v - 1)}. \qedhere
\end{equation*}
\end{proof}

\section{Almost perfect sequence covering arrays}\label{s:almost}

We now focus specifically on the case where $n=2$. The sequences now under consideration are those containing four points of $\textup{PG}(2,q)$. These sequences can divided into three families. The first family contains all sequences of four points such that no three are collinear. As hyperplanes and lines are the same in $\textup{PG}(2,q)$, these sequences are frames. By \lref{l:framecover}, in any permutation representation $\Psi \le \sym_{r}$ of $\textrm{PGL}(3,q)$ as defined in Section 3, any frame is covered by $\vert \Psi \vert/24$ permutations. The second family contains all sequences with three collinear points but not four. The third family contains all sequences of four collinear points. 

Let $r = q^{2} + q + 1$ be the number of points in $\textrm{PG}(2,q)$. A \textit{difference set} of $\mathbb{Z}_{r}$ is a set $A = \{ a_{0}, a_{1}, \dots, a_{q} \} \subset \mathbb{Z}_{r}$ such that for any non-zero element $x \in \mathbb{Z}_{r}$, there are $i$ and $j$ such that $a_{i} - a_{j} = x$. If $A = \{a_{0}, \dots, a_{q}\}$ is a difference set of $\mathbb{Z}_{r}$, then the set $\{ \{ a_{i} + j : i \in [q+1] \} : j \in \mathbb{Z}_{r} \}$, where addition is performed modulo $r$, is the lineset of an incidence structure isomorphic to $\textrm{PG}(2,q)$. For any prime power $q$, it is possible to label the points of $\textrm{PG}(2,q)$ such that the labelled lineset has the above form \cite{Sin37}. Let $\psi$ be such a labelling. In this section we show that in the corresponding permutation representation $\Psi \leqslant \sym_{r}$ of $\textup{PGL}(3,q)$, each sequence containing at most three collinear points is covered by $\vert \Psi \vert/24$ permutations. This in turn proves \tref{t:collineationgroup}.

The consequence of this result is that for fixed $v$, we can find a much smaller multiset of permutations than the multiset built in Section 3 such that the vast majority of sequences are covered by a constant number of permutations. This prompts two immediate questions. The first is whether the bound in \tref{t:upperbound} can be reduced specifically when $t=4$. Indeed, in Section 3, our construction relied on reducing the point set to a subset where no three points lay on the same hyperplane. Here, it seems we can relax that condition to one that requires that no four points lie on the same line. A \textit{$(k,d)$-arc} in $\textrm{PG}(2,q)$ is a set of $k$ points in $\textrm{PG}(2,q)$, no $d+1$ of which lie on the same line. Thus, we can form a PSCA of strength 4 by taking the representation $\Psi$ and deleting all but a subset of symbols that form a $(k,3)$-arc. Unfortunately, the maximum size of a $(k,3)$-arc is $2q + 3$ \cite{Hir79}. This is approximately twice the size of the $(k,2)$-arc used in the construction in Section 3. As a result, even if a $(2q+3,3)$-arc exists, using such an arc only reduces the bound in \tref{t:upperbound} by a factor of approximately $2^{8}$.

The second question is whether the permutation representation $\Psi$ may actually form a PSCA by covering sequences containing four collinear points with the same constant number of permutations. A computer search was performed over difference sets of $\mathbb{Z}_{r}$ for orders $5 \leqslant q \leqslant 25$. This search found no representations of the full collineation group $\pgaml(3,q)$ of $\textup{PG}(2,q)$ that also formed a PSCA of strength 4. Examples of representations of $\pgaml(3,2)$, $\pgaml(3,3)$ and $\pgaml(3,4)$ in $\sym_{7}$, $\sym_{13}$ and $\sym_{21}$ respectively that form PSCAs were recorded by Gentle and Wanless \cite{GeWa22}. Note that $\pgaml(3,2) = \textup{PGL}(3,2)$ and $\pgaml(3,3) = \textup{PGL}(3,3)$.

The remainder of this section is devoted to a series of lemmas that collectively prove \tref{t:collineationgroup}. Consider sequences that contain three collinear points but not four. For $i \in \{1, 2, 3, 4\}$, let $T_{i}$ be the set of sequences $s$ of four points in $\textrm{PG}(2,q)$ for which there exists a line $\ell$ such that $s_{i}$ does not lie on $\ell$, but $s_{j}$ does for $j \in \{1, 2, 3, 4\} \backslash \{i \}$. Then, for distinct $i$ and $j$, it is impossible to find a collineation of $\textrm{PG}(2,q)$ that maps a sequence in $T_{i}$ to a sequence in $T_{j}$. However, we have the following lemma.
\begin{lemma}
If $s$ and $s'$ are sequences in $T_{i}$ for some $i \in \{1,2,3,4\}$ then there exists a projectivity of $\textup{PG}(2,q)$ that maps $s$ to $s'$.
\end{lemma}

\begin{proof}
Let $\{i,j,k,\ell\} = \{1,2,3,4\}$ and let $s_{i} = a$, $s_{j} = b$, $s_{k} = c$, $s_{\ell} = d$ and $s'_{i} = a'$, $s'_{j} = b'$, $s'_{k} = c'$, $s'_{\ell} = d'$. Hence, both $\{ a, b, c \}$ and $\{ a', b', c' \}$ are sets of non-collinear points. As such, the non-zero vectors $u \in a$, $v \in b$ and $w \in c$ form a basis of $V$, as do the non-zero vectors $u' \in a'$, $v' \in b'$ and $w' \in c'$. We can find a matrix $A \in \textrm{GL}(3, q)$ such that $Au = u'$, $Av = v'$ and $Aw = w'$. Let $f$ be the projectivity of $\textrm{PG}(2,q)$ induced by $A$. Then $f(a) = a'$, $f(b) = b'$ and $f(c) = c'$. Let $x \in d$. Since $b,c$ and $d$ are collinear, $x = \alpha v + \beta w$ for some non-zero $\alpha, \beta \in \textrm{GF}(q)$. Hence, $Ax = \alpha v' + \beta w'$. Let $x' \in d'$ and similarly note that $x' = \alpha' v' + \beta' w'$ for some non-zero $\alpha',\beta' \in \textrm{GF}(q)$. We can then find a matrix $B \in \textrm{GL}(3, q)$ such that $Bu' = u'$, $Bv' = \alpha^{-1}\alpha'v'$ and $Bw' = \beta^{-1}\beta'w'$. Then, $BAx = \alpha'v' + \beta'w' = x'$. Let $g$ be the projectivity induced by $B$ and observe $g \circ f (a) = a'$, $g \circ f (b) = b'$, $g \circ f (c) = c'$ and $g \circ f (d) = d'$. Therefore, the projectivity $g \circ f$ maps $s$ to $s'$.
\end{proof}

Now we choose some bijection $\psi : \textrm{PG}(2,q) \rightarrow [r]$ and consider the corresponding permutation group $\Psi \le \sym_{r}$. Let $L$ be the subsets of $[r]$ corresponding to lines of $\textrm{PG}(2,q)$. Then each line $\ell \in L$ can be represented by $\ell = \{\ell_{0}, \dots, \ell_{q}\} \subseteq [r]$ where $\ell_{0} < \ell_{1} < \dots < \ell_{q}$. Our next goal is to find $\vert \asc(s) \vert$ for $s \in T_{i}$. First consider $T_{1}$. Let $\ell \in L$ and let $i \in [q+1]$. There are $\binom{q - i}{2}$ 3-subsets of $\ell$ that have $\ell_{i}$ as their minimal element and there are $\ell_{i} - i$ points less than $\ell_{i}$ that do not lie on $\ell$. Therefore, for $s \in T_{1}$,
\begin{equation}
\vert \asc(s) \vert = \sum_{\ell \in L} \sum_{i = 0}^{q} \binom{q - i}{2} (\ell_{i} - i).  \label{e:asct1} \tag{\textup{E1}}
\end{equation}

Next, consider $T_{2}$. Suppose we build a sequence in $T_{2}$ containing three points on $\ell$ such that $\ell_{i}$ is the minimum of these points. Then, for $j > i$, there are $q - j$ points on $\ell$ greater than $\ell_{j}$ and $(\ell_{j} - j) - (\ell_{i} - i)$ points between $\ell_{i}$ and $\ell_{j}$ that do not lie on $\ell$. Therefore, for $s \in T_{2}$,
\begin{align}
\vert \asc(s) \vert &=  \sum_{\ell \in L}  \sum_{0 \leq i < j \leq q} (q - j)((\ell_{j} - j) -  (\ell_{i} - i))  \nonumber\\
&= \sum_{\ell \in L} \left( \sum_{j = 1}^{q} \sum_{i = 0}^{j-1} (q - j)\ell_{j} - \sum_{j=1}^{q}\sum_{i=0}^{j-1} (q-j)j - \sum_{i=0}^{q-1}\sum_{j=i+1}^{q} (q-j)\ell_{i} + \sum_{i=0}^{q-1}\sum_{j=i+1}^{q} (q-j)i \right) \nonumber\\
&= \sum_{\ell \in L} \left( \sum_{j=1}^{q} \left(j(q-j)\ell_{j} - j^{2}(q-j)\right) - \sum_{i=0}^{q-1} \left( \binom{q-i}{2}\ell_{i} - \binom{q-i}{2}i \right) \right) \nonumber\\
&= \sum_{\ell \in L} \sum_{i=0}^{q} \left( i(q-i) - \binom{q-i}{2} \right)(\ell_{i} - i). \label{e:asct2} \tag{\textup{E2}}
\end{align}

Next, consider $T_{3}$. Suppose we build a sequence in $T_{3}$ containing three points on $\ell$ such that $\ell_{j}$ is the maximum of these points. Then, for $i < j$, there are $i$ points on $\ell$ less than $\ell_{i}$ and $(\ell_{j} - j) - (\ell_{i} - i)$ points between $\ell_{i}$ and $\ell_{j}$ that do not lie on $\ell$. Therefore, for $s \in T_{3}$,
\begin{align}
\vert \asc(s) \vert &=  \sum_{\ell \in L}  \sum_{0 \leq i < j \leq q} i((\ell_{j} - j) -  (\ell_{i} - i))  \nonumber\\
&= \sum_{\ell \in L} \left( \sum_{j = 1}^{q} \sum_{i = 0}^{j-1} i\ell_{j} - \sum_{j=1}^{q}\sum_{i=0}^{j-1} ij - \sum_{i=0}^{q-1}\sum_{j=i+1}^{q} i\ell_{i} + \sum_{i=0}^{q-1}\sum_{j=i+1}^{q} i^{2} \right) \nonumber\\
&= \sum_{\ell \in L} \left( \sum_{j=1}^{q} \left( \binom{j}{2}\ell_{j} - \binom{j}{2}j \right) - \sum_{i=0}^{q-1} \left( i(q-i)\ell_{i} - i^{2}(q - i) \right) \right) \nonumber\\
&= \sum_{\ell \in L} \sum_{i=0}^{q} \left( \binom{i}{2} - i(q-i) \right)(\ell_{i} - i). \label{e:asct3} \tag{\textup{E3}}
\end{align}

Finally, consider $T_{4}$. Let $\ell \in L$ and let $i \in [q+1]$. There are $\binom{i}{2}$ 3-subsets of $\ell$ whose maximum is $\ell_{i}$ and there are $q^{2} + q - \ell_{i} - (q-i)$ points greater $\ell_{i}$ that do not lie on $\ell$. Therefore, for $s \in T_{4}$,
\begin{align}
\vert \asc(s) \vert = \sum_{\ell \in L}  \sum_{i = 0}^{q} \binom{i}{2} (q^{2} - (\ell_{i} - i)).  \label{e:asct4} \tag{\textup{E4}}
\end{align}
Therefore,
\begin{align}
\eqref{e:asct1} + \eqref{e:asct2} + \eqref{e:asct3} + \eqref{e:asct4} &= \sum_{\ell \in L} \sum_{i = 0}^{q} \binom{i}{2}q^{2} \nonumber\\
&= \frac{(q^{2} + q + 1)q^{3}(q + 1)(q - 1)}{6}. \label{e:asctot} \tag{\textup{E5}}
\end{align}

Later in this section, we will prove the existence of a particular representation $\Psi$ such that \eref{e:asct1} = \eref{e:asct2} = \eref{e:asct3} = \eref{e:asct4}. Our first step is to equate \eref{e:asct1} + \eref{e:asct4} and \eref{e:asct2} + \eref{e:asct3} for which we require the following lemma.

\begin{lemma}\label{l:linesum}
Let $L$ be the lineset of $\textup{PG}(2,q)$ as subsets of $[r]$. Then,
\begin{equation*}
\sum_{\ell \in L} \sum_{i = 0}^{q} i\ell_{i} = \frac{(q^{2} + q )(q^{2} + q + 1)(2q^{2} + 2q + 1)}{6}
\end{equation*}
\end{lemma}

\begin{proof}
Let $j \in [r]$ and let $J = \{ (i, \ell) : \ell_{i} = j \}$. Consider $\sum_{(i,\ell) \in J} i \ell_{i}$. For each pair $(i, \ell) \in J$, $\ell_{i} = j$ and $i$ is the number of points less than $j$ that lie on $\ell$. There are $j$ elements of $[r]$ less than $j$, each of which must lie on the same line as $j$ exactly once. Therefore,
\begin{equation*}
\sum_{(i,\ell) \in J} i \ell_{i} = j \sum_{(i,\ell) \in J} i = j^{2}.
\end{equation*}
Therefore, 
\begin{equation*}
\sum_{\ell \in L} \sum_{i = 0}^{q} i\ell_{i} = \sum_{j = 0}^{q^{2} + q} j^{2} = \frac{(q^{2} + q)(q^{2} + q + 1)(2q^{2} + 2q + 1)}{6} \qedhere
\end{equation*}

\end{proof}

\begin{lemma}\label{l:ascsum}
In any representation $\Psi$, \eref{e:asct1} $+$ \eref{e:asct4} $=$ \eref{e:asct2} $+$ \eref{e:asct3} $=$ \eref{e:asctot}/$2$.
\end{lemma}
\begin{proof}
Using the expressions \eref{e:asct2} and \eref{e:asct3},
\begin{align*}
\eqref{e:asct2} + \eqref{e:asct3} &= \sum_{\ell \in L} \sum_{k = 0}^{q} \left( \binom{k}{2} - \binom{q - k}{2} \right)(\ell_{k} - k)\\
&= \frac{1}{2} \sum_{\ell \in L} \sum_{k = 0}^{q} \left(2k(q - 1) - (q^{2} - q) \right)(\ell_{k} - k).
\end{align*}
First, consider
\begin{align*}
\sum_{\ell \in L} \sum_{k = 0}^{q}  (q^{2} - q)(\ell_{k} - k).
\end{align*}
The sum $\sum_{\ell \in L} \sum_{k = 0}^{q} \ell_{k}$ is just the sum of all the points of every line of $\textrm{PG}(2,q)$. Each point appears on $q + 1$ lines so this is equal to $\binom{q^{2} + q + 1}{2}(q + 1)$. Next, $\sum_{\ell \in L}\sum_{k = 0}^{q} k = (q^{2} + q + 1)\binom{q + 1}{2}$. Therefore,
\begin{align*}
\sum_{\ell \in L} \sum_{k = 0}^{q}  (q^{2} - q)(\ell_{k} - k) &= (q^{2} - q) \left( \binom{q^{2} + q + 1}{2}(q + 1) - (q^{2} + q + 1)\binom{q + 1}{2} \right)\\
&= (q^{2} + q + 1)(q^{2} - q)(q + 1) \left( \frac{q^{2} + q}{2} - \frac{q}{2} \right)\\
&= \frac{1}{2}(q^{2} + q + 1)(q^{3} - q)q^{2}.
\end{align*}
Next, consider
\begin{equation*}
\sum_{\ell \in L} \sum_{k = 0}^{q} k (\ell_{k} - k).
\end{equation*}
By \lref{l:linesum}, this is equal to
\begin{equation*}
\frac{1}{6} \left( (q^{2} + q)(q^{2} + q + 1)(2q^{2} + 2q + 1) - (q^{2} + q + 1)q(q + 1)(2q + 1) \right) = \frac{(q^{2} + q + 1)(q^{2} + q)q^{2}}{3}.
\end{equation*}
Therefore,
\begin{align*}
\eqref{e:asct2} + \eqref{e:asct3} &= \frac{1}{2}\left( \frac{2(q^{2} + q + 1)(q^{3} - q)q^{2}}{3} - \frac{(q^{2} + q + 1)(q^{3} - q)q^{2}}{2} \right)\\
&= \frac{(q^{2} + q + 1)(q^{3} - q)q^{2}}{12}\\
&= \eqref{e:asctot}/2.
\end{align*}
As \eref{e:asct1} + \eref{e:asct2} + \eref{e:asct3} + \eref{e:asct4} = \eref{e:asctot}, it must also be true that \eref{e:asct1} + \eref{e:asct4} = \eref{e:asctot}/2.
\end{proof} 

\lref{l:ascsum} applies generally to any representation of $\textrm{PGL}(3,q)$, but now we will choose a specific representation according to a difference set of $\mathbb{Z}_{r}$. Recall that a difference set is a set $A = \{ a_{0}, a_{1}, \dots, a_{q} \} \subset \mathbb{Z}_{r}$ such that for any non-zero element $x \in \mathbb{Z}_{r}$, there are $i$ and $j$ such that $a_{i} - a_{j} = x$. For the rest of this section, we choose $\psi$ such that the lines of $\textrm{PG}(2,q)$ are mapped to sets of the form $\{ a_{i} + j : i \in [q+1] \}$ for $j \in \mathbb{Z}_{r}$ and a difference set $A = \{a_{0}, \dots, a_{q}\}$ where $a_{i} < a_{i+1}$ for $i \in \{0, \dots, q-1 \}$. Furthermore, we can assume without loss of generality that $a_{0} = 0$. Let $A + j$ denote the line $\{ a_{i} + j : i \in [q+1] \}$. For any integer $k$, define $a_{k(q + 1) + i} := a_{i}$. 

Our goal is to show that for this particular choice of $\psi$, in the corresponding permutation subgroup $\Psi \le \sym_{r}$, \eref{e:asct1} = \eref{e:asct2} = \eref{e:asct3} = \eref{e:asct4}. This will be done by proving that for $i \in [q+1]$
\begin{equation*}
\sum_{\ell \in L} (q^{2} + q - \ell_{q - i}) = \sum_{\ell \in L} \ell_{i}.
\end{equation*}
To do so, we first require the following lemma.

\begin{lemma}\label{l:diffset}
For the difference set $A = \{ a_{0}, a_{1}, \dots, a_{q} \} \subset \mathbb{Z}_{r}$, and for $i \in [q+1]$,
\begin{equation}
\sum_{k = 0}^{q} (a_{k + i} - a_{k})(a_{k} - a_{k - 1}) = \sum_{k = 0}^{q} (a_{k + 1} - a_{k})(a_{k} - a_{k - i}) \label{e:diffset}
\end{equation} 
\end{lemma}

\begin{proof}
\begin{align*}
\sum_{k = 0}^{q} (a_{k + i} - a_{k})(a_{k} - a_{k - 1}) &= \sum_{k = 0}^{q} \left(a_{k + i}a_{k} - a_{k + i}a_{k - 1} - a_{k}^{2} + a_{k}a_{k - 1}\right)\\
&= \sum_{k=0}^{q} \left( a_{k}a_{k-i} - a_{k+1}a_{k-i} - a_{k}^{2} + a_{k+1}a_{k}\right)\\
&= \sum_{k = 0}^{q} (a_{k + 1} - a_{k})(a_{k} - a_{k-i}). \qedhere
\end{align*}

\end{proof}

\begin{lemma}
Let $L = \{A + j: j \in \mathbb{Z}_{r} \}$ for the difference set $A = \{ a_{0}, \dots, a_{q} \}$. For $i \in [q+1]$
\begin{equation}
\sum_{\ell \in L} (q^{2} + q - \ell_{q - i}) = \sum_{\ell \in L} \ell_{i}. \label{e:lineflip}
\end{equation}
\end{lemma}

\begin{proof}
First, we consider $\sum_{\ell \in L} \ell_{i}$ for $i \in [q+1]$. Consider the values of $j$ for which $a_{k} + j$ is the smallest element of $A + j$. First, when $j = q^{2} + q + 1 - a_{k}$, the smallest element of $A + j$ is $a_{k} + j = 0$. Then, when $j = q^{2} + q + 1 - a_{k - 1} - 1$, the smallest element of $A + j$ is $a_{k} + j = a_{k} - a_{k - 1} - 1$ and the largest element of $A + j$ is $a_{k - 1} + j = q^{2} + q$. Hence, the smallest element of $A + (j + 1)$ would be $a_{k - 1} + (j + 1) = 0$. So, for $j \in [q^{2} + q + 1 - a_{k}, q^{2} + q - a_{k - 1}]$, the smallest point on the line $A + j$ is $a_{k} + j$. The $i$th smallest point on these lines will therefore be $a_{k + i} + j$ and will range in value from $a_{k + i} - a_{k}$ to $a_{k + i} - a_{k - 1} - 1$. The sum of the integers in this interval is
\begin{equation*}
\binom{a_{k} - a_{k - 1}}{2} + (a_{k + i} - a_{k})(a_{k} - a_{k - 1}).
\end{equation*}
Therefore,
\begin{equation*}
\sum_{\ell \in L} \ell_{i} = \sum_{k = 0}^{q} \left( \binom{a_{k} - a_{k - 1}}{2} + (a_{k + i} - a_{k})(a_{k} - a_{k - 1}) \right).
\end{equation*}

Similarly, the range of values for $j$ for which $a_{k} + j$ is the largest point of the line $A + j$ is the interval $[ q^{2} + q + 1 - a_{k + 1}, q^{2} + q + 1 - a_{k} - 1]$. For these lines, the $i$th largest point is $a_{k - i} + j$ which ranges in value from $q^{2} + q + 1 + a_{k - i} - a_{k + 1}$ to $q^{2} + q + a_{k - i} - a_{k}$. Therefore,
\begin{align*}
\sum_{\ell \in L} (q^{2} + q - \ell_{q - i}) &= \sum_{k=0}^{q} \sum_{j = q^{2} + q + 1 - a_{k + 1}}^{q^{2} + q  - a_{k}} (q^{2} + q - (a_{k - i} + j))\\
&= \sum_{k=0}^{q} \sum_{j = 0}^{a_{k + 1} - a_{k} - 1} (a_{k} - a_{k - i} + j)\\
&= \sum_{k=0}^{q} \left(\binom{a_{k + 1} - a_{k}}{2} + (a_{k + 1} - a_{k})(a_{k} - a_{k - i})\right).
\end{align*} 
By \lref{l:diffset},
\begin{equation*}
\sum_{\ell \in L} (q^{2} + q - \ell_{q - i}) =  \sum_{k = 0}^{q} \left( \binom{a_{k + 1} - a_{k}}{2} + (a_{k + i} - a_{k})(a_{k} - a_{k - 1}) \right) = \sum_{\ell \in L} \ell_{i}. \qedhere
\end{equation*}

\end{proof}

Now we can prove that for this choice of $\Psi$, \eref{e:asct1} = \eref{e:asct2} = \eref{e:asct3} = \eref{e:asct4}.
\begin{lemma}\label{l:equalasc}
Let $\psi: \textup{PG}(2,q) \rightarrow [r]$ be a bijection such that the lines of $\textup{PG}(2,q)$ are mapped to the sets $\{A + j: j \in \mathbb{Z}_{r} \}$ for the difference set $A = \{ a_{0}, \dots, a_{q} \}$. Then, $\eqref{e:asct1} = \eqref{e:asct2} = \eqref{e:asct3} = \eqref{e:asct4} = \eqref{e:asctot}/4$.
\end{lemma}

\begin{proof}
We can substitute \eref{e:lineflip} into \eref{e:asct4} and obtain
\begin{align*}
\eqref{e:asct4} &= \sum_{\ell \in L} \left( \sum_{i = 0}^{q} \binom{i}{2} (q^{2} - (\ell_{i} - i)) \right)\\
&= \sum_{\ell \in L} \left( \sum_{i = 0}^{q} \binom{i}{2} (q^{2} + i - (q^{2} + q - \ell_{q - i}) ) \right)\\
&= \sum_{\ell \in L} \left( \sum_{i = 0}^{q} \binom{i}{2} (\ell_{q - i} - (q - i)  ) \right)\\
&= \sum_{\ell \in L} \left( \sum_{i = 0}^{q} \binom{q - i}{2} (\ell_{i} - i  ) \right)\\
&= \eqref{e:asct1}.
\end{align*}
By \lref{l:ascsum}, \eref{e:asct1} + \eref{e:asct4} = \eref{e:asctot}/2. Therefore, \eref{e:asct1} = \eref{e:asct4} = \eref{e:asctot}/4. We can also substitute \eref{e:lineflip} into \eref{e:asct3} and find
\begin{align*}
\eqref{e:asct3} &= \sum_{\ell \in L} \left( \sum_{0 \leq i < j \leq q} i((\ell_{j} - j) - (\ell_{i} - i)) \right)\\
&= \sum_{\ell \in L} \left( \sum_{0 \leq i < j \leq q} i((q^{2} + q - \ell_{q - j} - j) - (q^{2} + q - \ell_{q - i} - i)) \right)\\
&= \sum_{\ell \in L} \left( \sum_{0 \leq i < j \leq q} i((\ell_{q - i} - (q - i)) -  (\ell_{q - j} - (q - j))) \right)\\
&= \sum_{\ell \in L} \left( \sum_{0 \leq i < j \leq q} (q - j)((\ell_{j} - j) - (\ell_{i} - i)) \right)\\
&= \eqref{e:asct2}.
\end{align*}
Again, by \lref{l:ascsum}, \eref{e:asct2} + \eref{e:asct3} = \eref{e:asctot}/2. Therefore, \eref{e:asct2} = \eref{e:asct3} = \eref{e:asctot}/4.
\end{proof}

We are now ready to prove \tref{t:collineationgroup}

\begin{proof}[Proof of Theorem \ref{t:collineationgroup}]
If $\psi$ is of the form outlined in \lref{l:equalasc}, then under the action of $\Psi$ on $\sym_{r,4}$, and for $i \in \{1, 2, 3, 4\}$ and $s \in T_{i}$, 
\begin{equation*}
\vert \asc(s) \vert = \frac{\eqref{e:asctot}}{4} = \frac{(q^{2} + q + 1)q^{3}(q + 1)(q - 1)}{24}.
\end{equation*}
Now consider $\stab(s)$ and note that $\vert \stab(s) \vert = \vert \textrm{PGL}(3, q) \vert/\vert T_{i} \vert$. There are $(q+1)q(q-1)$ 3-sequences that can be formed from points of a given line $\ell$. There are then $q^{2}$ points not on $\ell$. Therefore,
\begin{align*}
\vert T_{i} \vert = (q^{2} + q + 1)q^{3}(q+1)(q-1).
\end{align*}
Thus,
\begin{equation*}
\vert \stab(s) \vert \vert \asc(s) \vert = \frac{\vert \Psi \vert}{(q^{2} + q + 1)q^{3}(q+1)(q-1)} \frac{(q^{2} + q + 1)q^{3}(q + 1)(q - 1)}{24} = \frac{\vert \Psi \vert}{24}.
\end{equation*}
Thus, by \lref{l:ascstab}, every sequence in $T_{i}$ for $i \in \{1,2,3,4\}$ is covered by $\vert \Psi \vert / 24$ permutations in $\Psi$. By \lref{l:framecover}, we also know that every frame in $\sym_{r,4}$ is covered by $\vert \Psi \vert/24$ permutations in $\Psi$. Moreover, by \tref{t:fundamental}, the number of frames in $\sym_{r,4}$ is exactly 
\begin{equation*}
\vert \textup{PGL}(3,q)\vert = (q^{2} + q + 1)q^{3}(q+1)(q-1)^{2}.
\end{equation*} 
Adding this to the number of sequences in $T_{i}$ for $i \in \{1,2,3,4\}$ we find that the number of sequences in $\sym_{r,4}$ that are covered by exactly $\vert \Psi \vert /24$ permutations in $\Psi$ is at least
\begin{equation*}
(q^{2} + q + 1)q^{3}(q+1)(q-1)(q+3).
\end{equation*}
We divide this number by $\vert \sym_{r,4} \vert$ and conclude 
\begin{align*}
\frac{(q^{2} + q + 1)q^{3}(q+1)(q-1)(q+3)}{\vert \sym_{r,4} \vert} &= \frac{(q^{2} + q + 1)q^{3}(q+1)(q-1)(q+3)}{(q^{2} + q + 1)(q^{2} + q)(q^{2} + q - 1)(q^{2} + q - 2)}\\
&= \frac{q^{2}(q + 3)}{(q^{2} + q - 1)(q + 2)}\\
&> \frac{q}{q+1}.
\end{align*}
This completes the proof.
\end{proof}

\section*{Acknowledgements}

The author is grateful to Daniel Horsley, Ian Wanless, Jonathan Jedwab, Shuxing Li and Jingzhou Na for useful discussions. The author was supported by an Australian Government Research Training Program (RTP) Scholarship.

\end{document}